\documentclass[a4paper,12pt]{amsart}

\usepackage{mathtools}
\usepackage{amsfonts}
\usepackage{amssymb}
\usepackage{amsthm,enumerate}
\usepackage{amsrefs}
\usepackage{tikz}
\usetikzlibrary{decorations.markings}

\newcommand{\R}{\mathbb{R}}

\newcommand{\Z}{\mathbb{Z}}

\newcommand{\inn}[1]{f^{\mathrm{in}}_{#1}}
\newcommand{\out}[1]{f^{\mathrm{out}}_{#1}}

\DeclareMathOperator{\head}{head}
\DeclareMathOperator{\tail}{tail}

\long\def\symbolfootnote[#1]#2{\begingroup
\def\thefootnote{\fnsymbol{footnote}}\footnote[#1]{#2}\endgroup}

\newtheorem{theorem}{Theorem}[section]

\newtheorem{lemma}[theorem]{Lemma}
\newtheorem{thm}[theorem]{Theorem}
\newtheorem{prop}[theorem]{Proposition}
\newtheorem{cor}[theorem]{Corollary}

\theoremstyle{remark}
\newtheorem{rmk}[theorem]{Remark}

\theoremstyle{definition}
\newtheorem{defn}[theorem]{Definition}

\DeclareFontFamily{OT1}{pzc}{}
\DeclareFontShape{OT1}{pzc}{m}{it}{<-> s * [1.1] pzcmi7t}{}
\DeclareMathAlphabet{\mathpzc}{OT1}{pzc}{m}{it}

\title{Arcs on Spheres Intersecting At Most Twice}

 \author{Christopher Smith}
  \author{Piotr Przytycki$^{\dag}$}
\address{Dept. of Math. \& Stats.\\
                    McGill University \\
                    Montreal, Quebec, Canada H3A 0B9}

\email{christopher.smith6@mail.mcgill.ca}
\email{piotr.przytycki@mcgill.ca}
\thanks{$\dag$ Partially supported by NSERC, FRQNT, National Science Centre DEC-2012/06/A/ST1/00259, and UMO-2015/\-18/\-M/\-ST1/\-00050.}
\begin{document}
\maketitle
\begin{abstract}
\noindent Let $p$ be a puncture of a punctured sphere, and let $Q$ be the set of all other punctures. We prove that the maximal cardinality of a set $\mathcal{A}$ of arcs pairwise intersecting at most once, which start at $p$ and end in $Q$, is $|\chi|(|\chi|+1).$ We deduce that the maximal cardinality of a set of arcs with arbitrary endpoints pairwise intersecting at most twice is $|\chi|(|\chi|+1)(|\chi|+2).$
\end{abstract}

\section{Introduction}

Let $S$ be a punctured sphere with Euler characteristic $\chi < 0$. We consider collections of essential simple arcs $\mathcal{A}$ on $S$ that are pairwise non-homotopic.

\begin{thm} \label{thm:twice} The maximal cardinality of a set $\mathcal{A}$ of arcs pairwise intersecting at most twice is $$|\chi|(|\chi|+1)(|\chi|+2).$$
\end{thm}

\begin{rmk}
\label{rmk:conj}
We conjecture that the formula also holds for any connected, oriented surface with $\chi < 0$. The proof we provide here, however, applies only to the case of spheres.
\end{rmk}

We will reduce Theorem~\ref{thm:twice} to the following:

\begin{thm} \label{thm:once} Let $p$ be a puncture of $S$, and $Q$ be the set of all other punctures. The maximal cardinality of a set $\mathcal{A}$ of arcs pairwise intersecting at most once, which start at $p$ and end in $Q$, is $$|\chi|(|\chi|+1).$$
\end{thm}

\medskip\noindent \textbf{Previous Results.} Initially, cardinality questions were asked about sets $\mathcal{C}$ of essential and nonperipheral simple closed curves on an arbitrary surface. Juvan, Malni\v{c} and Mohar proved that given $k$ bounding the number of intersections of any two curves in $\mathcal{C}$, there is an upper bound on $|\mathcal{C}|$ \cite{JMM}. For $k$ fixed, a polynomial upper bound of order $k^2+k+1$ in $|\chi|$ was obtained in \cite[Cor 1.6]{P}. This was recently greatly improved by Aougab, Biringer and Gaster \cite{ABG} to an upper bound of order $\frac{|\chi|^{3k}}{(\log |\chi|)^2}$. They further proved that if the genus of the surface is fixed, and we vary the number of punctures $n$, then $|\mathcal {C}|\leq \mathcal{O}(n^{2k+2})$.

Farb and Leininger asked for the precise asymptotics in the case where $k=1$. For $S$ a torus, we have $|\mathcal{C}|\leq 3$.
For $S$ a closed genus 2 surface, Malestein, Rivin, and Theran proved that the maximal cardinality of
$\mathcal{C}$ is 12
\cite{MRT}. They also produced a lower quadratic and exponential upper bound on maximal $|\mathcal{C}|$ in terms of $|\chi|$, which was obtained independently by Farb and Leininger \cite{L}. This was later improved to a cubic upper bound in \cite[Thm 1.4]{P} and recently to an upper bound of order $\frac{|\chi|^3}{(\log |\chi|)^2}$ by Aougab, Biringer and Gaster \cite{ABG}.

It seems that corresponding questions about arcs are easier to tackle. By \cite[Thm 1.2]{P}, the maximal cardinality of a set $\mathcal{A}$ of arcs pairwise intersecting at most once is $2|\chi|(|\chi|+1)$. If $S$ is a punctured sphere, then by \cite[Thm 1.7]{P}, given two distinguished (but not necessarily distinct) punctures $p, p'$, the maximal cardinality of a set $\mathcal{A}$ of arcs pairwise intersecting at most once which start at $p$ and end at $p'$ is
${\frac12 |\chi|(|\chi|+1)}$. Theorem~\ref{thm:once} is a very useful addition to this, but note that it is much more difficult to prove that any of the theorems in \cite{P}.

Recently, A.\ Bar-Natan proved that if in Theorem~\ref{thm:twice} the arcs in $\mathcal A$ are required to start and end at a distinguished puncture $p$, then the maximal cardinality of $\mathcal{A}$ is $\frac{1}{6}|\chi|(|\chi|+1)(|\chi|+2)$ \cite{B}.

\medskip\noindent \textbf{Organisation.} In Section 2 we construct collections of arcs of appropriate size satisfying the conditions of Theorems~\ref{thm:twice} and \ref{thm:once}. We also provide a counter-example showing that Theorem~\ref{thm:once} does not generalise to a 3-punctured torus. In Section~3 we show how Theorem~\ref{thm:once} implies Theorem~\ref{thm:twice}. In Section 4 we provide the bulk of the definitions and proof of Theorem~\ref{thm:once} up to two propositions which are proven in Sections~5 and 6.

\medskip\noindent \textbf{Acknowledgement.} We thank the referee for many comments that greatly helped to improve the exposition.

\section{Examples}

\subsection{Preliminaries}

Given a punctured sphere $S$ (or, in Subsection~\ref{subs:notgeneral}, a 3-punctured torus), an \emph{arc} on $S$ is a map from $(0,1)$ to~$S$ that is proper. A proper map induces a map between topological ends of spaces, and in this sense each endpoint of $(0,1)$ is sent to a puncture of~$S$. We will say that the arc \emph{starts} and \emph{ends} at these punctures.
An arc is \emph{simple} if it is an embedding. In that case we can and will identify the arc with its image in $S$. A \emph{homotopy} between arcs $\alpha$ and~$\beta$ is a proper map $(0,1)\times [0,1]\to S$ which restricts to $\alpha$ on $(0,1)\times{0}$ and to~$\beta$ on $(0,1)\times{1}$. In particular, $\alpha$ and~$\beta$ start at the same puncture and end at the same puncture. We often identify the punctured sphere $S$ with the interior of the punctured disc~$D$ whose boundary circle $\partial D$ corresponds to a distinguished puncture~$p$ of $S$. Then each arc ending (or starting) at $p$ can be homotoped to an arc that limits to a point on $\partial D$, and we will be only considering such arcs. Note that homotopic arcs ending at $p$ might limit to different points of~$\partial D$. We occasionally also identify $S$ with a punctured annulus~$A$ whose boundary circles correspond to $p$ and another distinguished puncture~$r$, where we also restrict to arcs with limit points.

An arc $\alpha$ is \emph{essential} if it cannot be homotoped into a puncture in the sense that there is no proper map $(0,1)\times [0,1)\to S$ restricting to $\alpha$ on $(0,1)\times{0}$. Unless otherwise stated, all arcs in the article are simple and essential.

We say that arcs $\alpha$ and $\beta$ are in \emph{minimal position}, if the number of intersection points $|\alpha\cap\beta|$ cannot be decreased by a homotopy. A \emph{bigon} (respectively, \emph{half-bigon}) between arcs $\alpha$ and $\beta$ is an embedded closed disc $B\subset S$ (respectively, properly embedded half-disc $B=[0,1]\times [0,1)\subset S$) such that $B\cap (\alpha\cup \beta)=\partial B$ and both $\partial B\cap \alpha$ and $\partial B\cap \beta$ are connected. It is a well known fact, which we will frequently use, that $\alpha$ and $\beta$ are in minimal position if and only if they are transverse and there is no bigon or half-bigon between them.

\subsection{Lower bound in Theorem~\ref{thm:once}}
First, we show an example of a collection of $|\chi|(|\chi|+1)$ arcs satisfying the conditions of Theorem~\ref{thm:once}. We identify the $n$-punctured sphere with the interior of a punctured disc whose boundary corresponds to the distinguished puncture~$p$, and whose other punctures $Q$ lie on a smaller circle in the interior of the disc (see Figure~\ref{fig:disjoint}). Note that $|Q| = n-1$. From each puncture $r \in Q$, we draw $n-2$ disjoint arcs that are straight line segments dividing the disc into $n-2$ regions, each containing one puncture in $Q - \{r\}$. We obtain $(n-1)(n-2)$ arcs that pairwise intersect at most once. As $|\chi| = n - 2$, we have
$(n-1)(n-2)=|\chi|(|\chi|+1)$.

\begin{figure}[h]
\centering
\begin{tikzpicture}
\draw (2,2) circle (2cm);

\draw (1.3,1.3) -- (2,4);
\draw (1.3,1.3) -- (4,2);
\draw (1.3,1.3) -- (0.6,0.6);

\draw[fill=white] (1.3,1.3) circle (0.05cm);
\draw[fill=white] (1.3,2.7) circle (0.05cm);
\draw[fill=white] (2.7,2.7) circle (0.05cm);
\draw[fill=white] (2.7,1.3) circle (0.05cm);
\end{tikzpicture}
\caption{On the 5-punctured sphere, for each $r \in Q$ we draw three disjoint arcs, for a total of 12 arcs.}
\label{fig:disjoint}
\end{figure}
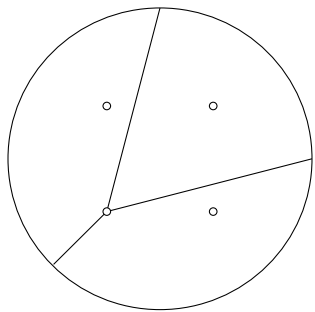

\subsection{Theorem \ref{thm:once} does not generalise}
\label{subs:notgeneral}

Now we provide a counterexample demonstrating that Theorem \ref{thm:once} does not hold in the case of a 3-punctured torus. We model the 3-punctured torus as a regular ideal octagon with edges identified as shown in Figure~\ref{fig:octagon}. After this identification we have three punctures, $\{a, b, p\}$ with $Q = \{a, b\}$, and we consider arcs connecting~$a$ or~$b$ to~$p$.

After the identifications in Figure~\ref{fig:octagon}, two of the edges connect $p$ to the puncture $a$; these will be the first two arcs in our collection. To these we add the 8 diagonals which run from the vertices labelled~$a$ or~$b$ to one of the vertices labelled $p$. Finally we add the three depicted arcs to the collection, for a total of 13 arcs. Since $|\chi| = 3$, the formula from Theorem~\ref{thm:once} evaluates to 12, showing that this theorem does not hold in general for non-spheres.

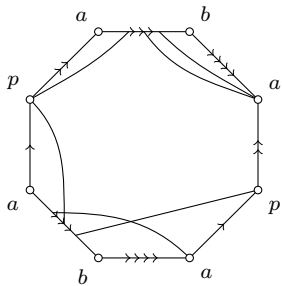
\begin{figure}[h]
\centering
\begin{tikzpicture}

\draw[postaction={decoration={markings, mark=at position 0.4 with {\arrow{>}}, mark=at position 0.5 with {\arrow{>}}, mark=at position 0.6 with {\arrow{>}}},decorate}]
	(0.9,3) -- (2.1,3);
\draw[postaction={decoration={markings, mark=at position 0.35 with {\arrow{>}}, mark=at position 0.45 with {\arrow{>}}, mark=at position 0.55 with {\arrow{>}}, mark=at position 0.65 with {\arrow{>}}},decorate}]
	(2.1,3) -- (3,2.1);
\draw[postaction={decoration={markings, mark=at position 0.45 with {\arrow{>}}, mark=at position 0.55 with {\arrow{>}}},decorate}]
	(3,0.9) -- (3,2.1);
\draw[postaction={decoration={markings, mark=at position 0.5 with {\arrow{>}}},decorate}]
	(2.1,0) -- (3,0.9);
\draw[postaction={decoration={markings, mark=at position 0.35 with {\arrow{>}}, mark=at position 0.45 with {\arrow{>}}, mark=at position 0.55 with {\arrow{>}}, mark=at position 0.65 with {\arrow{>}}},decorate}]
	(0.9,0) -- (2.1,0);
\draw[postaction={decoration={markings, mark=at position 0.4 with {\arrow{>}}, mark=at position 0.5 with {\arrow{>}}, mark=at position 0.6 with {\arrow{>}}},decorate}]
	(0,0.9) -- (0.9,0);
\draw[postaction={decoration={markings, mark=at position 0.5 with {\arrow{>}}},decorate}]
	(0,0.9) -- (0,2.1);
\draw[postaction={decoration={markings, mark=at position 0.45 with {\arrow{>}}, mark=at position 0.55 with {\arrow{>}}},decorate}]
	(0,2.1) -- (0.9,3);

\draw (3,2.1) edge[out=150,in=315] (1.7,3) ;
\draw (0.6,0.3) -- (3,0.9);

\draw (3,2.1) edge[out=160,in=305] (1.5,3);
\draw (0.45,0.45) edge[out=90,in=-45] (0,2.1);

\draw (0,2.1) edge[out=30,in=-135] (1.3,3);
\draw (0.3,0.6) edge[out=0,in=135] (2.1,0);

\draw[fill=white] (0,2.1) circle (0.05cm) node [above left] {\tiny $p$};
\draw[fill=white] (0.9,3) circle (0.05cm) node [above left] {\tiny $a$};
\draw[fill=white] (2.1,3) circle (0.05cm) node [above right] {\tiny $b$};
\draw[fill=white] (3,2.1) circle (0.05cm) node [above right] {\tiny $a$};
\draw[fill=white] (3,0.9) circle (0.05cm) node [below right] {\tiny $p$};
\draw[fill=white] (2.1,0) circle (0.05cm) node [below right] {\tiny $a$};
\draw[fill=white] (0.9,0) circle (0.05cm) node [below left] {\tiny $b$};
\draw[fill=white] (0,0.9) circle (0.05cm) node [below left] {\tiny $a$};
\end{tikzpicture}
\caption{The 11th to 13th arcs lying on a 3-punctured torus. Notice that no diagonal from either $a$ or $b$ to $p$ twice intersects any of these three arcs.}
\label{fig:octagon}
\end{figure}

\subsection{Lower bound in Theorem~\ref{thm:twice}}

Finally, we show an example of a collection of $|\chi|(|\chi|+1)(|\chi|+2)$ arcs satisfying the conditions of Theorem~\ref{thm:twice}. To this end consider an $n$-punctured sphere constructed in the following fashion. Let $P$ be an ideal $n$-gon, and glue two copies of $P$ together along their corresponding edges. We will think of the two copies of $P$ as the {\it front} and {\it back faces} of the sphere.

Our collection $\mathcal{A}$ will consist of three types of arcs. First, let $\mathcal{E}$ be the set of all $n$ edges along which the polygons are glued. Second, let $\mathcal{D}$ be the set of the $\frac{n(n-3)}{2}$ diagonals between vertices on each of the two faces, for a total of $n(n-3)$ arcs of this form.

Finally, we form a set $\mathcal{C}$ of arcs which cross between faces of the sphere. For each edge $e$ of the polygons, consider the midpoint $m$ of~$e$, and every ordered pair $(u, v)$ of punctures outside~$e$, with $u$ and $v$ not necessarily distinct. For each such choice of $e$ and $(u,v)$, we include in~$\mathcal{C}$ the arc whose first half is the straight line segment from $u$ to $m$ on the front face, and whose second half is the straight line segment from $m$ to $v$ on the back face. See Figure~\ref{fig:c}. For each edge $e$, there are $(n-2)^2$ such arcs, and we have $n$ edges, giving us a total of $n(n-2)^2$ arcs of this form.

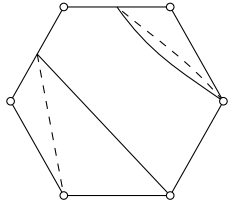
\begin{figure}[h]
\centering
\begin{tikzpicture}
\draw (0.7,2.5) -- (2.1,2.5) -- (2.8,1.25) -- (2.1,0) -- (0.7,0) -- (0,1.25) -- cycle;

\draw (2.1,0) -- (0.35,1.875);
\draw[dashed] (0.35,1.875) -- (0.7,0);

\draw (2.8,1.25) edge[out=145,in=-55] (1.4,2.5);
\draw[dashed] (2.8,1.25) edge[out=130,in=-40] (1.4,2.5);

\draw[fill=white] (0.7,2.5) circle (0.05cm);
\draw[fill=white] (2.1,2.5) circle (0.05cm);
\draw[fill=white] (2.8,1.25) circle (0.05cm);
\draw[fill=white] (2.1,0) circle (0.05cm);
\draw[fill=white] (0.7,0) circle (0.05cm);
\draw[fill=white] (0,1.25) circle (0.05cm);
\end{tikzpicture}
\caption{An example of two of the arcs in $\mathcal{C}$ on a 6-punctured sphere. The dashed lines indicate where the arc lies on the back face.}
\label{fig:c}
\end{figure}

Recalling that $n = |\chi| + 2$ we obtain that $n + n(n-3) + n(n-2)^2 = |\chi|(|\chi|+1)(|\chi|+2)$.

It remains to verify that these arcs satisfy the conditions of Theorem~\ref{thm:twice}. It is clear that each of the edges in $\mathcal{E}$ and the diagonals in $\mathcal{D}$ are homotopically distinct from one another, and as they are straight line segments on their respective faces, they can pairwise intersect at most once.

Consider the arcs in $\mathcal{C}$. First, we wish to show that arcs in $\mathcal{C}$ are homotopically distinct from each other, and from arcs in $\mathcal{D}$ or $\mathcal{E}$. As arcs in $\mathcal{C}$ each necessarily intersect exactly one of the edges, the latter is clear. If we have two arcs in $\mathcal{C}$ with the same $(u,v)$ but differing midpoint $m$, they must be homotopically distinct, because we know that these arcs must intersect the edge $e$, and do not intersect any other edge.

Now consider the case where two arcs in $\mathcal{C}$ have $u, v$, and $m$ in common. Consider Figure~\ref{fig:cut-pentagram}. We cut along the edges other than $e$ to obtain a disc. The two arcs intersect exactly once without creating any half-bigons. As a half-bigon on the $n$-punctured sphere would survive the cutting, we conclude that the two arcs cannot be homotopic. So all arcs in $\mathcal{C}, \mathcal{D}$, and $\mathcal{E}$ are homotopically distinct.

\begin{figure}
\centering
\begin{tikzpicture}
\draw (2,0.4) -- (0.8,0) -- (0,1) -- (0.8,2) -- (2,1.6);
\draw (2,0.4) -- (3.2,0) -- (4,1) -- (3.2,2) -- (2,1.6);
\draw [dashed] (2,0.4) -- (2,1.6);
\draw (0.8,2) -- (3.2,0); 
\draw (3.2,2) -- (0.8,0); 

\draw[fill=white] (0.8,0) circle (0.05cm) node [below] {\tiny $u$};
\draw[fill=white] (3.2,0) circle (0.05cm) node [below] {\tiny $u$};
\draw[fill=white] (2,0.4) circle (0.05cm) node [below] {\tiny $a$};
\draw[fill=white] (2,1.6) circle (0.05cm) node [above] {\tiny $b$};
\draw[fill=white] (0.8,2) circle (0.05cm) node [above] {\tiny $v$};
\draw[fill=white] (3.2,2) circle (0.05cm) node [above] {\tiny $v$};
\draw[fill=white] (0,1) circle (0.05cm) node [left] {\tiny $c$};
\draw[fill=white] (4,1) circle (0.05cm) node [right] {\tiny $c$};
\draw[fill=black] (2,1) circle (0.05cm) node [right] {\tiny $m$};
\end{tikzpicture}
\caption{The disc obtained by cutting the 5-punctured sphere along each edge other than $e$.} \label{fig:cut-pentagram}
\end{figure}
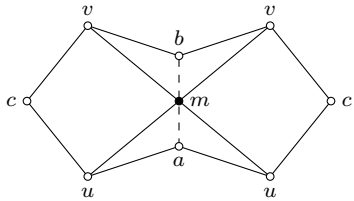

Finally, we wish to show that arcs in $\mathcal{C}$ pairwise intersect at most twice with arcs in $\mathcal{C} \cup \mathcal{D} \cup \mathcal{E}$. Note that the arcs in $\mathcal{D}$ are diagonals, each lying on exactly one face of $S$, and the arcs in $\mathcal{C}$ each consist of two straight line segments, one lying on each face. As such, two arcs in $\mathcal{C}$ may pairwise intersect at most twice (at most once on each face), and arcs in $\mathcal{D}$ and $\mathcal{E}$ pairwise intersect at most once with arcs in $\mathcal{C}$.

\begin{rmk}
Since any surface of Euler characteristic $\chi<0$ is obtained by appropriately gluing two ideal $|\chi|+2$ gons, the analogous construction of $\mathcal{A}$ gives lower bound in Remark~\ref{rmk:conj}.
\end{rmk}

\section{Reduction to Theorem~\ref{thm:once}}
\label{sec:reduction}

In this section we adapt a proof from \cite{P} and use Theorem~\ref{thm:once}, whose proof is postponed, in order to obtain the upper bound on $|\mathcal{A}|$ in Theorem~\ref{thm:twice}. We begin with some definitions.

We equip $S$ with an arbitrary complete hyperbolic metric, and we realise all arcs as geodesics. It is well known that they are then pairwise in minimal position. Around each puncture the metric is that of a hyperbolic cusp, inside which the arcs are disjoint and appear in a cyclic order.

\begin{defn}[{\cite[Def 2.1]{P}}] A {\it tip} $\tau$ of $\mathcal{A}$ is a pair $(\alpha,\beta)$ of oriented arcs in $\mathcal{A}$ starting at the same puncture and consecutive in the cyclic order. That is to say that there is no other arc in $\mathcal{A}$ issuing from this puncture in the clockwise oriented cusp sector from $\alpha$ to $\beta$.

Let $\tau = (\alpha, \beta)$ be a tip and let $N_\tau$ be an open abstract ideal hyperbolic triangle with vertices $a, t, b$. The tip $\tau$ determines a unique local isometry $\nu_\tau : N_\tau \to S$ sending $ta$ to $\alpha$ and $tb$ to $\beta$ and mapping a neighbourhood of $t$ to the clockwise oriented cusp sector from $\alpha$ to $\beta$. We call $\nu_\tau$ the {\it nib} of $\tau$. See Figure~\ref{fig:nib}.
\end{defn}

\begin{prop}\label{prop:piotr} Suppose that the arcs in $\mathcal{A}$ pairwise intersect at most twice. Let $\nu : N = \sqcup_\tau N_\tau \to S$ be the disjoint union of all the nibs $\nu_\tau$. Then for each point $p \in S$, the preimage $\nu^{-1}(p)$ consists of at most $(|\chi|+1)(|\chi|+2)$ points.
\end{prop}

In order to prove Proposition~\ref{prop:piotr}, we require some further results from \cite{P}.

\begin{defn}[{\cite[Def 2.4]{P}}] Let $n \in N_\tau$ be a point in the domain of a nib. The {\it slit} at $n$ is the restriction of $\nu_\tau$ to the geodesic ray in $N_\tau$ joining $t$ with $n$. See Figure~\ref{fig:nib}.\end{defn}

\begin{figure}
\begin{center}
\includegraphics[width=0.35\textwidth]{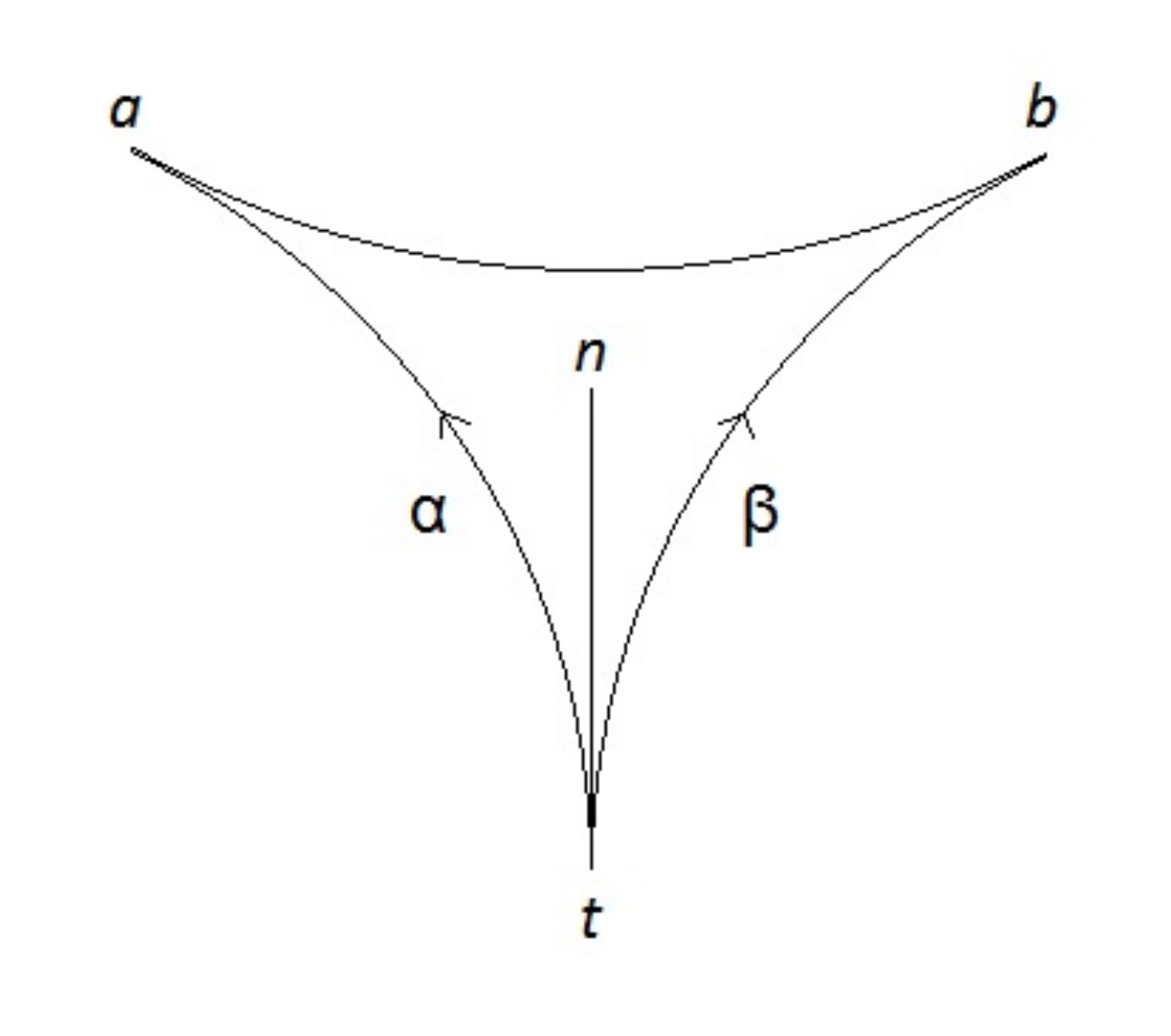}
\end{center}
\caption{A tip, its nib and a slit}
\label{fig:nib}
\end{figure}

\begin{lemma}[{\cite[Lem 2.5]{P}}] \label{lem:slits_embed} A slit is an embedding. \end{lemma}

\begin{lemma}[{\cite[Lem 3.2]{P}}]\label{lemma:slits} Suppose that the arcs in $\mathcal{A}$ pairwise intersect at most $k \geq 1$ times. If for distinct $n, n' \in N$ we have $\nu(n) = \nu(n')$, then the images in $S$ of slits at $n, n'$ intersect at most $k-1$ times outside the endpoint.\end{lemma}

\begin{proof}[Proof of Proposition~\ref{prop:piotr}] Let $p \in S$ be an arbitrary point. Let $S'$ be the sphere obtained from $S$ by introducing a puncture at $p$. By Lemma~\ref{lem:slits_embed}, the slit at any $n\in \nu^{-1}(p)$ embeds in $S$, so it is a simple arc on $S'$. By Lemma~\ref{lemma:slits}, for any two points $n, n' \in \nu^{-1}(p)$, the slits at $n$ and $n'$ intersect at most once. Therefore on $S'$, we have a collection of simple arcs which start at $p$ and end at punctures of $S$, and these arcs pairwise intersect at most once.

If $\chi$ is the Euler characteristic of $S$, then $\chi - 1$ is the Euler characteristic of $S'$, and so by Theorem~\ref{thm:once} the maximal size of a collection of such slits is $|\chi-1|(|\chi-1| + 1) = (|\chi|+1)(|\chi|+2)$.  Since each point in the preimage $\nu^{-1}(p)$ contributes an arc to this collection, this gives us the desired bound on the size of $\nu^{-1}(p)$.
\end{proof}

\begin{proof}[Proof of Theorem~\ref{thm:twice}] Each arc in $\mathcal{A}$ is the first arc of exactly two tips, depending on its orientation, so the area of $N$ is $2|\mathcal{A}|\pi$. The area of the punctured sphere $S$ is $2\pi|\chi|$. By Proposition~\ref{prop:piotr}, the map $\nu : N \to S$ is at most $(|\chi|+1)(|\chi|+2)$-to-1, and so we have $$2|\mathcal{A}|\pi \leq 2\pi|\chi|(|\chi|+1)(|\chi|+2).$$
\end{proof}

\section{Outline of the proof of Theorem~\ref{thm:once}}

For the remainder of the paper, let $p$ be the distinguished puncture of the punctured sphere $S$. We will represent $S$ as the interior of a punctured disc $D$ with boundary associated with $p$, and with a collection of remaining punctures $Q$. Let $\mathcal{A}$ be as stated in Theorem~\ref{thm:once}. We will consider the arcs to be oriented from a puncture in $Q$ to the puncture~$p$. We assume without loss of generality that arcs in $\mathcal{A}$ are pairwise in minimal position.

If for each puncture $r \in Q$ the arcs in $\mathcal{A}$ starting at $r$ are disjoint, then Theorem~\ref{thm:once} is easy to prove: there can be at most $|Q|-1$ arcs in $\mathcal{A}$ from each puncture on $D$ (refer back to Figure~\ref{fig:disjoint}), and there are $|Q|$ punctures on the disc, giving us $|Q| \cdot (|Q|-1) = |\chi|\cdot(|\chi|+1)$ arcs, recalling $|\chi| = |Q|-1$. It makes sense then to begin by considering some properties of arcs starting at the same puncture~$r$ which intersect.

\subsection{Fish} In this subsection we will introduce our main tool to account for intersecting arcs starting at the same puncture $r$.

\begin{figure}[h]
\centering
\begin{tikzpicture}
\draw (2,2) circle (2cm);

\draw plot [smooth, tension=0.7] coordinates {(2,2) (2.5,2.3) (3,2) (3.5,1.5) (3.85,1.3)};
\node at (2.5,2.3) [above] {\tiny $\beta$};
\draw plot [smooth, tension=0.7] coordinates {(2,2) (2.5,1.7) (3,2) (3.5,2.5) (3.85, 2.7)};
\node at (2.5,1.7) [below] {\tiny $\alpha$};

\node at (4,2) [right] {\tiny $p$};
\node at (2,2) [left] {\tiny $r$};

\draw[fill=white] (2,2) circle (0.05cm);
\draw[fill=white] (2.5,2) circle (0.05cm);
\draw[fill=white] (3.5,2) circle (0.05cm);
\end{tikzpicture}
\caption{A fish $(\alpha, \beta)$.}
\label{fig:fish}
\end{figure}
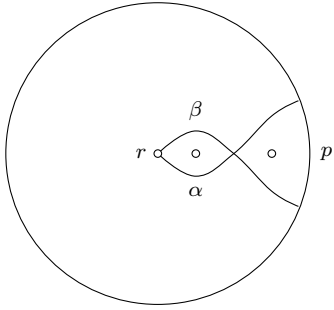

\begin{defn}
Let $r \in Q$. A {\em fish} $F$ with {\em nose} $r$ is a pair of arcs $\alpha, \beta \in \mathcal{A}$ from $r$ to $p$ which intersect. See Figure~\ref{fig:fish}.

The arcs $\alpha$ and $\beta$ divide the disc $D$ into three regions. The {\em head} of~$F$, denoted by $\head(F)$, is the region adjacent to $r$, but not to $p$. The {\em tail} of $F$, denoted by $\tail(F)$, is the region adjacent to $p$, but not to $r$. The remaining region is thought of as being {\em outside} the fish. We write $F=(\alpha,\beta)$ as an ordered pair whenever $\head(F)$ covers the counterclockwise oriented cusp sector between $\alpha$ and $\beta$.

We denote by $h(F)$ the subset of the punctures $Q$ which are in $\head(F)$, and by $t(F)$ the set of punctures which are in $\tail(F)$. We do not consider the nose of the fish to be contained in the head. If $q \in t(F)$, then we say that $F$ is a $q$-{\em fish}.
\end{defn}

Note that if $S$ has only $3$ punctures, then there are no fish.

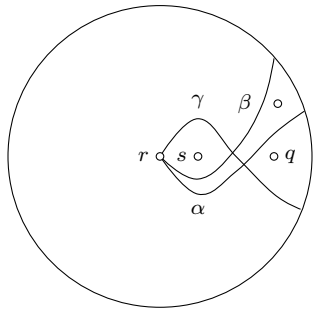
\begin{figure}[h]
\begin{tikzpicture}
\draw (2,2) circle (2cm);

\draw plot [smooth, tension=0.7] coordinates {(2,2) (2.5,2.5) (3,2) (3.5,1.5) (3.85,1.3)};
\node at (2.5,2.5) [above] {\tiny $\gamma$};

\draw plot [smooth, tension=0.7] coordinates {(2,2) (2.5,1.7) (3,2.1) (3.35,2.7) (3.5, 3.3)};
\node at (3.35,2.7) [left] {\tiny $\beta$};

\draw plot [smooth, tension=0.7] coordinates {(2,2) (2.5, 1.5) (3,1.8) (3.5, 2.3) (3.9, 2.6)};
\node at (2.5,1.5) [below] {\tiny $\alpha$};

\draw[fill=white] (2,2) circle (0.05cm) node [left] {\tiny $r$};
\draw[fill=white] (2.5,2) circle (0.05cm) node [left] {\tiny $s$};
\draw[fill=white] (3.5,2) circle (0.05cm) node [right] {\tiny $q$};
\draw[fill=white] (3.55,2.7) circle (0.05cm);
\end{tikzpicture}
\caption{A minimal fish $(\beta, \gamma)$ and a non-minimal fish $(\alpha, \gamma)$. In this picture we have $r \sim_q s$.}
\label{fig:non-minimal}
\end{figure}

\begin{defn}
For each puncture $r \in Q$, consider the collection of arcs in $\mathcal{A}$ starting at $r$. Recall that we assign these arcs a cyclic order (based on their intersection points with a sufficiently small circle about the puncture $r$). Here and in the remainder of the article, unlike in Section~\ref{sec:reduction}, we order them counterclockwise (this convention gives better figures).

A fish $F=(\alpha, \beta)$ with nose $r$ is {\em minimal} if $\alpha$ and $\beta$ are consecutive in the cyclic order around $r$. Equivalently, in $\mathcal A$ there is no arc from $r$ which begins in $\head(F)$.

Denote by $\mathcal{F}^r$ the set of minimal fish with nose $r$, and by $\mathcal{F}_q$ the set of minimal $q$-fish.
\end{defn}

\begin{defn}
Let $q \in Q$. The equivalence relation $\sim_q$ on $Q-\{q\}$ is the equivalence relation generated by setting $r \sim_q s$ whenever there is a $q$-fish $F$ (not necessarily minimal) with nose $r$ such that $s \in h(F)$ (see Figure~\ref{fig:non-minimal}). Define $c_q$ to be the number of equivalence classes of $\sim_q$.
\end{defn}

\subsection{Proof of Theorem \ref{thm:once}}

Recall our previous discussion that the interesting case for proving Theorem~\ref{thm:once} was when arcs from the same puncture $r$ intersect. Clearly if we allow arcs from $r$ to intersect one another, we will be able to include more arcs from $r$ in our collection. We can think of this as a positive contribution to $|\mathcal{A}|$ at $r$. What we will show is that this positive contribution is balanced out by a matching negative contribution.

Each of these extra arcs from $r$ will cause punctures in $Q - \{r\}$ to lie in the tails of fish with nose $r$. If a puncture $q$ lies in the tail of some fish, then this will reduce the number of equivalence classes of $\sim_q$. In turn, this will reduce the number of arcs in our collection which may begin at $q$.

The following lemma, propositions, and corollary can be thought of as proving that each positive contribution at one puncture must be matched by a negative contribution at some combination of other punctures.

We postpone the proofs of Propositions~\ref{lemma:classes} and~\ref{lemma:bigons} to Sections~\ref{section:classes} and~\ref{section:bigons} respectively. Here we start by deducing Corollary~\ref{cor:bigons} from Proposition~\ref{lemma:bigons}. We then show that Proposition~\ref{lemma:classes} and Corollary~\ref{cor:bigons} imply Theorem~\ref{thm:once}.

\begin{prop} \label{lemma:classes}
Let $q \in Q$. Then $$|\chi| \geq |\mathcal{F}_q| + c_q.$$
\end{prop}

\begin{prop} \label{lemma:bigons}
Let $r \in Q$. If $\mathcal{A}$ contains $k_r$ arcs from $r$ to $p$, then $$k_r \leq |\chi| + \sum_{F \in \mathcal{F}^r} |t(F)|.$$
\end{prop}

\begin{cor}\label{cor:bigons}
Under the hypotheses of Proposition \ref{lemma:bigons}, $$k_r \leq c_r + \sum_{F \in \mathcal{F}^r} |t(F)|.$$
\end{cor}

To prove this corollary, we first need the following lemma.

\begin{lemma} \label{lemma:tails}
For each $q$-fish $F$, no arc in $\mathcal A$ from $q$ to $p$ may pass through the region $\head(F)$.
\end{lemma}

\begin{proof}
Let $F = (\alpha, \beta)$, and let $\gamma \in \mathcal{A}$ be an arc from $q$ to $p$. Observe that $q \in \tail(F)$ by hypothesis. For the arc $\gamma$ to pass through $\head(F)$, it must first exit the tail, then enter the head, and finally exit the head, before arriving at $p$. This would cause a total of three intersections between $\gamma$ and either $\alpha$ or $\beta$, which is a contradiction. Note that even if $\gamma$ passes directly through the intersection point $\alpha \cap \beta$ from $\tail(F)$ to $\head(F)$, we still have two intersections there, and one more upon leaving $\head(F)$.
\end{proof}

\begin{proof}[Proof of Corollary~\ref{cor:bigons}]
By Lemma~\ref{lemma:tails} for $q=r$, no arcs from $r$ can pass through $\head(F)$ for any $r$-fish $F$. Consider the modified disc $D'$ obtained in the following way.  Remove $\head(F)$ from $D$ for each $r$-fish $F$, and let $D'$ be the connected component of the resulting surface which contains the boundary~$p$. From the definition of $\sim_r$ we see that this $D'$ will be homeomorphic to an at most $(c_r+2)$-punctured sphere: one way to see this is to observe that the punctures of each equivalence class of $\sim_r$ become a single puncture in $D'$, and the punctures $r$ and~$p$ each remain as well, giving at most $c_r+2$ total punctures. From Proposition~\ref{lemma:bigons}, it follows that $$k_r \leq |\chi(D')| + \sum_{F \in \mathcal{F}^r} |t(F)| \leq c_r + \sum_{F \in \mathcal{F}^r} |t(F)|.$$
\end{proof}

\begin{proof}[Proof of Theorem \ref{thm:once}]

We begin by summing the inequalities obtained in Proposition \ref{lemma:classes} and Corollary \ref{cor:bigons} over $Q$. Recall that $|Q| = |\chi| + 1$.  From Proposition \ref{lemma:classes} we get
$$|\chi|(|\chi|+1) = \sum_{q \in Q} |\chi| \geq \sum_{q \in Q} |\mathcal{F}_q| + \sum_{q \in Q} c_q.$$
 From Corollary \ref{cor:bigons} we get
$$|\mathcal{A}| \leq \sum_{r \in Q} c_r + \sum_{r \in Q} \sum_{F \in \mathcal{F}^r} |t(F)|.$$
Putting these together we obtain
\begin{align*}
|\chi|(|\chi|+1) \geq
	\sum_{q \in Q} |\mathcal{F}_q|
	- \sum_{r \in Q} \sum_{F \in \mathcal{F}^r} |t(F)|
	+ |\mathcal{A}|.
\end{align*}
Now observe that the first two terms on the right hand side are actually counting the same objects. The first term is counting for each puncture~$q$, the number of minimal fish with $q$ in their tail. The second term is counting for each minimal fish, the number of punctures in that fish's tail. These two values therefore cancel out, and we obtain $$|\chi|(|\chi|+1) \geq |\mathcal{A}|.\qedhere$$\end{proof}

\section{Equivalence Classes} \label{section:classes}

\subsection{Proof of Proposition~\ref{lemma:classes}}
We keep the notation from Section~4. The first supporting lemma has technical proof postponed to the next subsection.

\begin{lemma} \label{lemma:cycles}
Let $q \in Q$. Then there exists a puncture $v \in Q$, $v \neq q$ such that $v$ is not the nose of a $q$-fish.
\end{lemma}

\begin{lemma}
\label{lem:tailsdisjoint}
Let $q \in Q$. If $F,F'$ are $q$-fish with intersecting heads, then one of $F,F'$ has nose lying in the head of the other.
\end{lemma}
\begin{proof} Suppose by contradiction that the nose of each of $F,F'$ is outside the head of the other.
Let $x=\head(F)\cap \tail(F)$ and $y=\head(F')\cap\tail(F')$. If $x\notin \head(F')$ and $y\notin \head(F)$, then we have the configuration from Figure~\ref{fig:crossing-fish}. Since the figure already accounts for the intersections of all the arcs involved, the tails of $F,F'$ are disjoint, which is a contradiction.

If exactly one of $x,y$, say $x$, lies in the head of the other fish $F'$, then all the intersections between the arcs involved lie in $\head(F')$. Moreover, $y$ lies outside $\tail(F)\cap \head (F')$. Hence again the tails of $F,F'$ are disjoint, contradiction.

Since the arcs in $\mathcal A$ pairwise intersect at most once, we cannot have $x\in \head(F')$ and $y\in \head(F)$ (even if we allowed $F,F'$ to have disjoint tails).
\end{proof}

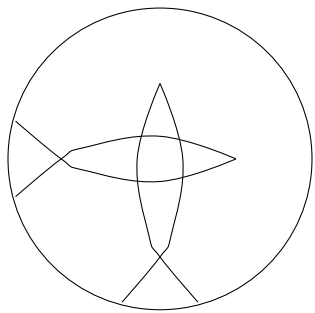
\begin{figure}
\begin{tikzpicture}
\draw (2,2) circle (2cm);

\draw plot [smooth, tension=0.7] coordinates {(2,3) (1.7,2) (1.85,1) (2,0.7) (2.5,0.1)};
\draw plot [smooth, tension=0.7] coordinates {(2,3) (2.3,2) (2.15,1) (2,0.7) (1.5,0.1)};

\draw plot [smooth, tension=0.7] coordinates {(3,2) (2,1.7) (1,1.85) (0.7,2) (0.1,2.5)};
\draw plot [smooth, tension=0.7] coordinates {(3,2) (2,2.3) (1,2.15) (0.7,2) (0.1,1.5)};

\end{tikzpicture}
\caption{The four intersections caused by the intersection of the heads prevent the tails of these fish from intersecting.} \label{fig:crossing-fish}
\end{figure}

To prove Proposition~\ref{lemma:classes} we argue by induction. First consider as the base case the three-punctured sphere where ${Q = \{q,r\}}$. Then the only equivalence class of $\sim_q$ is $\{r\}$ and there are clearly no fish possible, so $|\chi| = 1 = |\mathcal{F}_q| + c_q$, as desired.

Now consider the case $|\chi| \geq 2$. By Lemma~\ref{lemma:cycles}, there exists a puncture $v$ other than $q$ which is not the nose of a $q$-fish. To apply induction, define $\widetilde D$ to be $D$ with the puncture $v$ removed. Let $\tilde p=\partial \widetilde D$.
For each arc $\alpha\in \mathcal A$ that does not start at $v$, let $\tilde \alpha$ denote the arc that is the image of $\alpha$ under the embedding $D\subset \widetilde D$. Let $\widetilde{\mathcal{A}}$ be the family of all such~$\tilde \alpha$, possibly after a homotopy putting all of them pairwise in minimal position.

For a puncture $q$ of $D$ distinct from $v$, we denote by $\tilde q$ its image in~$\widetilde D$. Let $\widetilde Q$ be the set of all such $\tilde q$. By $\mathcal{F}_{\tilde q}$ we denote the set of minimal $\tilde q$-fish on $\widetilde D$ formed by the family~$\widetilde{\mathcal{A}}$. Let $\sim_{\tilde q}$ be the equivalence relation on the punctures of $\widetilde D$
generated by $\tilde r \sim_{\tilde q} \tilde s$ whenever there is a $\tilde q$-fish $\widetilde F$ formed by the family $\widetilde{\mathcal A}$, with nose $\tilde r$ such that $\tilde s \in h(\widetilde F)$. By $c_{\tilde q}$ we denote the number of equivalence classes of $\sim_{\tilde q}$. For the induction step it suffices to show that $|\mathcal{F}_{\tilde q}| + c_{\tilde q} \geq |\mathcal{F}_q| + c_q - 1$.

\begin{rmk}
\label{rem:refinement}
Observe that if for $q,r,s\neq v$ we have $r\not\sim_q s$, then certainly $\tilde r \not\sim_{\tilde q} \tilde s$, as the removal of $v$ does not introduce any new fish. In other words, under the obvious identification of $Q-\{v\}$ with $\widetilde{Q}$, the relation $\sim_{\tilde q}$ is a refinement of the restriction of $\sim_{q}$ to $Q-\{v\}$.
\end{rmk}

\begin{defn} Let $F=(\alpha, \beta)$ be a fish formed by arcs $\alpha,\beta\in \mathcal A$. We say that $F$ {\em vanishes} if $\tilde \alpha$ and $\tilde \beta$ are no longer in minimal position (they bound a half-bigon). Equivalently, $v$ is either the only tail puncture or the only head puncture of $F$.
\end{defn}

Note that if $F$ does not vanish, then $(\tilde \alpha,\tilde \beta)$ is a fish on $\widetilde D$ formed by the family~$\widetilde{\mathcal{A}}$. If $F$ is minimal, the fish $(\tilde \alpha,\tilde \beta)$ does not have to be minimal, since the arcs in $\widetilde{\mathcal{A}}$ might have been homotoped and their order around the nose might have changed. However, we have the following lemma, whose proof is also postponed to a separate subsection.

\begin{lemma}\label{lemma:passing}
For $q\neq v$, let $\mathcal{F}_q'$ be the set of non-vanishing minimal $q$-fish formed by the family $\mathcal{A}$. Then there is a one-to-one map $\varphi : \mathcal{F}_q' \to \mathcal{F}_{\tilde q}$.
\end{lemma}

The proof of Proposition~\ref{lemma:classes} splits into two cases:

\begin{enumerate}[(i)]
\item $\{v\}$ is an equivalence class of~$\sim_q$,
\item $\{v\}$ is not an equivalence class of~$\sim_q$.
\end{enumerate}

Let us first consider case (i). Since $\{v\}$ is an equivalence class of~$\sim_q$, by Remark~\ref{rem:refinement} we have $c_{\tilde q} \geq c_q - 1$.
Because $\{v\}$ is an equivalence class of $\sim_q$, we know that $v$ is not in the head of any $q$-fish. Furthermore, $v$ cannot be the only puncture in the tail of a $q$-fish, as every $q$-fish contains $q$ in its tail by definition. So no $q$-fish vanishes. Consequently, $|\mathcal{F}_q| = |\mathcal{F}'_q|$. By Lemma~\ref{lemma:passing}, we have $|\mathcal{F}'_q| \leq |\mathcal{F}_{\tilde q}|$.  Hence, $|\mathcal{F}_{\tilde q}| + c_{\tilde q} \geq |\mathcal{F}_q| + c_q - 1$ as required.

Now we consider case (ii). Since $\{v\}$ is not an equivalence class of~$\sim_q$, by Remark~\ref{rem:refinement} we have $c_{\tilde q} \geq c_q$.
Thus to prove Proposition~\ref{lemma:classes} it remains to prove that $|\mathcal{F}_{\tilde q}| \geq |\mathcal{F}_q| - 1$. This will follow from Lemma~\ref{lemma:passing} once we show that at most one minimal $q$-fish vanishes.

Suppose for contradiction that there exist two distinct vanishing minimal $q$-fish $F$ and $F'$. Let $F = (\alpha, \beta)$ and $F' = (\gamma, \delta)$. Then  $\{v\} = h(F) = h(F')$.

We consider two possibilities.

\begin{enumerate}[(a)]
\item $F$ and $F'$ have distinct noses,
\item $F$ and $F'$ have a common nose $r$.
\end{enumerate}

We first consider possibility (a). The nose of one fish cannot be contained in the head of the other, as each head contains only the puncture $v$.
This contradicts Lemma~\ref{lem:tailsdisjoint}.

It remains to consider possibility (b).

\begin{rmk}
\label{rmk:not_three}
Note that two distinct minimal fish whose tails intersect cannot be formed with only three arcs (one shared between them); three arcs forming two minimal fish $(\alpha, \beta)$ and $(\beta, \gamma)$ will always have disjoint tails (see Figure~\ref{fig:3-arcs}).
\end{rmk}

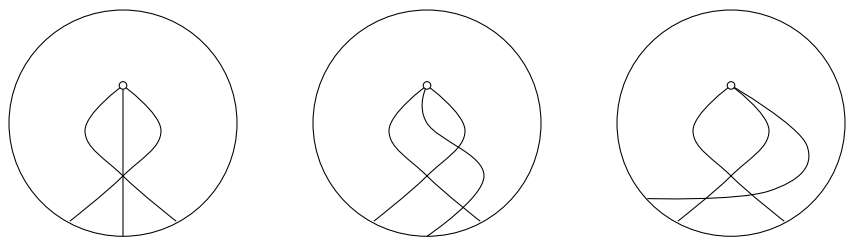
\begin{figure}[h!]
\begin{tikzpicture}
\draw (1.5,1.5) circle (1.5cm);
\draw (5.5,1.5) circle (1.5cm);
\draw (9.5,1.5) circle (1.5cm);

\draw plot [smooth, tension=0.7] coordinates {(1.5,2) (2,1.4) (1.5,0.8) (0.8,0.2)};
\draw plot [smooth, tension=0.7] coordinates {(1.5,2) (1,1.4) (1.5,0.8) (2.2,0.2)};
\draw (1.5,2) -- (1.5,0);

\draw plot [smooth, tension=0.7] coordinates {(5.5,2) (6,1.4) (5.5,0.8) (4.8,0.2)};
\draw plot [smooth, tension=0.7] coordinates {(5.5,2) (5,1.4) (5.5,0.8) (6.2,0.2)};
\draw plot [smooth, tension=0.7] coordinates {(5.5,2) (5.5,1.5) (6.25,0.8) (5.5,0)};

\draw plot [smooth, tension=0.7] coordinates {(9.5,2) (10,1.4) (9.5,0.8) (8.8,0.2)};
\draw plot [smooth, tension=0.7] coordinates {(9.5,2) (9,1.4) (9.5,0.8) (10.2,0.2)};
\draw plot [smooth, tension=0.7] coordinates {(9.5,2) (10.5,1.2) (10,0.6) (8.4,0.5)};

\draw[fill=white] (1.5,2) circle (0.05cm);
\draw[fill=white] (5.5,2) circle (0.05cm);
\draw[fill=white] (9.5,2) circle (0.05cm);
\end{tikzpicture}
\caption{If three consecutive arcs form two minimal fish, those fish have disjoint tails.}\label{fig:3-arcs}
\end{figure}

By Remark~\ref{rmk:not_three} in possibility (b) we can assume that $\alpha, \beta, \gamma, \delta$ are all distinct. Moreover, since $F,F'$ are minimal, the heads of $F,F'$ are disjoint in a neighbourhood of their common nose $r$. We can thus add a puncture $r'$ close to $r$ and move the head of $F'$ to have nose $r'$ without creating new intersections. This reduces possibility (b) to possibility (a), and finishes case (ii) and the entire proof of Proposition~\ref{lemma:classes}.

\subsection{Proof of Lemma~\ref{lemma:cycles}}

We start with a rough outline of the proof. If the conclusion of the lemma fails, there is a cycle of $q$-fish $F_i$, where the head of $F_i$ contains the nose of $F_{i+1}$. In Step~1 of the proof we will inscribe in this fish cycle an embedded polygon $P$ formed of the arcs $a_i$ in each $\head(F_i)$. In Steps~2 and~3 we will gain control on how the arcs $g_i$ forming $F_i$ exit $P$. Some of these arcs form regions $R_i$, where $q$ cannot lie (this will be shown in the final step of the proof). In Step~4 we will show that the union $R$ of $R_i$ contains the entire outside $E$ of $P$, hence $q$ cannot lie outside $P$. In Step~5 we will show that $q$ cannot lie inside~$P$. We now proceed with a proper proof.

\medskip
\noindent \textbf{Definition of $F_i$.} Suppose toward contradiction that for every puncture ${q_i \in Q - \{q\}}$ there is a $q$-fish $F_i$ with nose $q_i$. Every $F_i$ has at least one head puncture which, by hypothesis, is also the nose of a $q$-fish, so we can choose a finite sequence of such punctures $(q_1, \dotsc, q_k)$ such that $q_{i+1} \in h(F_i)$ and $q_1 \in h(F_k)$. Furthermore, we may choose that sequence to be one of minimal length. It can be easily seen that $k \neq 2$ (see Figure~\ref{fig:not-two}), so we have $k \geq 3$. As we are dealing with a cycle, it will be convenient to write $q_{k+1} = q_1$.

\begin{figure}[h!]
\begin{tikzpicture}
\draw (2,2) circle (2cm);

\draw plot [smooth, tension=0.5] coordinates {(1.5,2) (2.5,2.5) (3.25,2) (3.85, 1.3)};
\draw plot [smooth, tension=0.5] coordinates {(1.5,2) (2.5,1.5) (3.25,2) (3.85, 2.7)};
\draw plot [smooth, tension=0.5] coordinates {(2.5,2) (1.5,2.5) (0.75,2) (0.15, 1.3)};
\draw plot [smooth, tension=0.5] coordinates {(2.5,2) (1.5,1.5) (0.75,2) (0.15, 2.7)};

\draw[fill=white] (1.5,2) circle (0.05cm);
\draw[fill=white] (2.5,2) circle (0.05cm);
\end{tikzpicture}
\caption{Two fish cannot produce such a sequence, as their tails could not intersect.}\label{fig:not-two}
\end{figure}
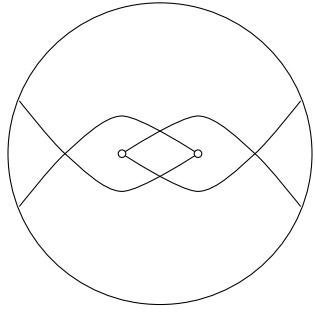

Since the sequence $(q_i)$ is of minimal length, $q_i \notin h(F_j)$ for all ${j \neq i-1}$. Note that if we remove all punctures other than $q$ and the~$q_i$ from~$D$, the punctures $(q_i)$ and the fish $(F_i)$ still give us a minimal length sequence of $q$-fish with the above properties, and so we may assume without loss of generality that there are no other punctures on~$D$. Since $q$ cannot be in the head of a $q$-fish, we have $h(F_i) = \{q_{i+1}\}$ for all $i$.

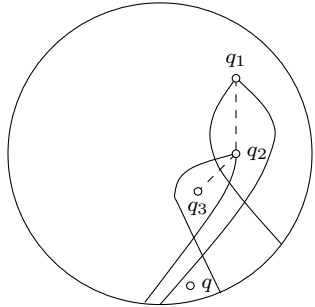
\begin{figure}[]
\begin{tikzpicture}
\draw (2,2) circle (2cm);

\draw plot [smooth, tension=0.7] coordinates {(3,3) (2.7,2.5) (2.7,2) (3,1.5) (3.6,0.8)};
\draw plot [smooth, tension=0.7] coordinates {(3,3) (3.45,2.5) (3.45,2) (3,1.2) (2,0)};

\draw plot [smooth, tension=0.7] coordinates {(3,2) (2.4,1.8) (2.2,1.5) (2.3,1.2) (2.8,0.15)};
\draw plot [smooth, tension=0.7] coordinates {(3,2) (2.8,1.4) (1.8,0.03)};

\draw [dashed] (3,3) -- (3,2.4)  -- (3,2);

\draw [dashed] (3,2) -- (2.5,1.5);

\draw[fill=white] (3,3) circle (0.05cm) node [above] {\tiny $q_1$};
\draw[fill=white] (3,2) circle (0.05cm) node [right] {\tiny $q_2$};
\draw[fill=white] (2.5,1.5) circle (0.05cm) node [below] {\tiny $q_3$};
\draw[fill=white] (2.4,0.25) circle (0.05cm) node [right] {\tiny $q$};
\end{tikzpicture}
\caption{The arcs $a_1$ and $a_2$ drawn as dashed lines.} \label{fig:a_i}
\end{figure}

For each $i$, let $a_i$ be an arc lying in $\head(F_i) - \head(F_{i+1})$ connecting $q_i$ to $q_{i+1}$ (see Figure~\ref{fig:a_i}), and in minimal position with respect to the arcs of $\mathcal{A}$. Such an arc exists, as $\head(F_i) - \head(F_{i+1})$ is connected.

\medskip

\noindent \textbf{Step 1.} \emph{We have ${\head(F_i) \cap \head(F_j) = \emptyset}$ for all ${i \neq j-1, j, j+1}$. In particular, all arcs $a_i$ are disjoint.}

\begin{proof}

For the second assertion, suppose $a_i$ and $a_j$ intersect. Since $a_j$ lies in $\head(F_j) - \head(F_{j+1})$, we have $i \neq j+1$. Similarly, $i \neq j-1$.
Thus to complete the proof, it remains to verify the first assertion. Since $q_i \notin h(F_j)$ and $q_j \notin h(F_i)$, the heads of $F_i$ and $F_j$ intersect in such a way that the nose of each is not contained in the head of the other. This contradicts Lemma~\ref{lem:tailsdisjoint}.

\end{proof}

\noindent \textbf{Definition of $P$.} By Step~1, the concatenation of the arcs $a_i$ forms an embedded ideal polygon in $D$ that we will call~$P$ (it is still possible that $q$ lies inside~$P$). For each $i$, the fish $F_i$ consists of the {\em outer} arc~$\out{i}$ which begins outside $P$, and the {\em inner} arc~$\inn{i}$ which begins inside $P$.
Without loss of generality, we may assume that $P$ is oriented as in Figure~\ref{fig:polygon}.

\begin{figure}
\begin{tikzpicture}
\draw (2,2) circle (2cm);

\draw plot [smooth, tension=0.7] coordinates {(2.7,2.7) (2.3,1.7) (2.7,0.8) (3.15,0.35)};
\draw plot [smooth, tension=0.7] coordinates {(2.7,2.7) (3.1,1.6) (2.7,0.7) (2.1,0)};

\draw [dashed] (2.7,2.7) -- (2.7,1.3) -- (1.3,1.3) -- (1.3,2.7) -- cycle;

\node at (2.4,1.8) [left] {\tiny $\inn{1}$};
\node at (3.0,1.6) [right] {\tiny $\out{1}$};

\node at (2.6,2) [below right] {\tiny $a_1$};
\node at (2,1.3) [below] {\tiny $a_2$};
\node at (1.3,2) [left] {\tiny $a_3$};
\node at (2,2.7) [above] {\tiny $a_4$};

\draw[fill=white] (2.7,2.7) circle (0.05cm) node [above right] {\tiny $q_1$};
\draw[fill=white] (2.7,1.3) circle (0.05cm) node [below] {\tiny $q_2$};
\draw[fill=white] (1.3,1.3) circle (0.05cm) node [below left] {\tiny $q_3$};
\draw[fill=white] (1.3,2.7) circle (0.05cm) node [above left] {\tiny $q_4$};

\end{tikzpicture}
\caption{Inner and outer arcs of a fish in the sequence.}\label{fig:polygon}
\end{figure}
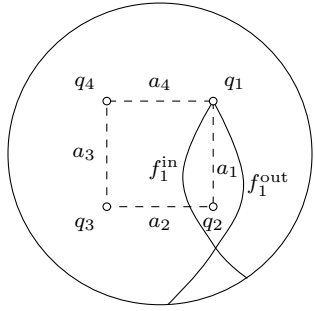

\medskip

\noindent \textbf{Step 2.} \emph{For any $i$, suppose that the first intersection of $\inn{i}$ with $P$ lies on $a_j$. Then $\inn{i}$ intersects $\inn{j}$ prior to its first point of intersection with $P$.}

\begin{proof}

First, recall that $a_i$ is contained in $\head(F_i)$, and hence the arcs which $F_i$ consists of do not intersect $a_i$. Therefore $j \neq i$. Additionally, $a_{i-1}$ is contained in $\head(F_{i-1}) - \head(F_i)$, and so the arcs which $F_i$ consists of do not intersect $a_{i-1}$. Therefore $j \neq i-1$. Thus $\inn{i}$ starts outside $\head(F_j)$.

Because the edge $a_j$ is contained entirely within $\head(F_j)$, crossing~$a_j$ requires crossing both $\inn{j}$ and $\out{j}$. Now refer to Figure~\ref{fig:crossing-direction}. For $\inn{i}$ to cross $a_j$ in the indicated direction, it is not possible for it to first cross $\out{j}$, then $a_j$, then $\inn{j}$. Therefore it must intersect $\inn{j}$ first, then $a_j$, and finally $\out{j}$.
\end{proof}

\begin{figure}
\begin{tikzpicture}
\draw (2,2) circle (2cm);

\draw plot [smooth, tension=0.7] coordinates {(2.8,1.55) (2.4,0.7) (2,0.7) (1.6,1.3) (1.4,2.5) (1.25,3.85)};

\draw (2.8,1.55) edge[out=145,in=0] (0,2);

\draw [dashed] (2,1) -- (1.2,1.55) -- (1.2,2.45) -- (2,3) -- (2.8,2.45) -- (2.8,1.55) -- cycle;

\node at (2.4,0.7) [below] {\tiny $\out{j}$};
\node at (2.2,1.9) [above] {\tiny $\inn{j}$};

\draw [thick,->] (2.295,1.425) -- (2.505,1.125);

\draw[fill=white] (2,1) circle (0.05cm);
\draw[fill=white] (1.2,1.55) circle (0.05cm);
\draw[fill=white] (1.2,2.45) circle (0.05cm);
\draw[fill=white] (2,3) circle (0.05cm);
\draw[fill=white] (2.8,2.45) circle (0.05cm);
\draw[fill=white] (2.8,1.55) circle (0.05cm) node [right] {\tiny $q_j$};

\end{tikzpicture}
\caption{Example: The dashed line $a_j$, marked with an arrow, cannot be crossed first in the indicated direction by an arc from any other puncture outside $\head(F_j)$ which intersects $\out{j}$ prior to $\inn{j}$.}\label{fig:crossing-direction}
\end{figure}
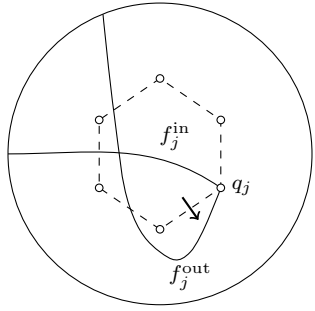

\noindent \textbf{Definition of $g_i$.} For each arc $\inn{i}$, consider the edge $a_j$ of $P$ with which it has its first intersection, and define a map $s(i) = j$. Then there exists a minimal length sequence $j_1, \dotsc, j_m$ with $s(j_i) = j_{i+1}$ and $s(j_m) = j_1$. Define $g_i = \inn{j_i}$. By Step~2 these $g_i$ each intersect the successive $g_{i+1}$ prior to leaving $P$.

\medskip
\noindent \textbf{Step 3.} \emph{$g_i$ intersects $g_{i-1}$ prior to leaving $P$.}

\begin{proof}

Consider Figure~\ref{fig:intersect-in-P}. The arc $g_{i-1}$ divides $\head(F_{j_i})$ into two regions, one of which is adjacent to $q_{j_i}$ and does not contain $q_{j_i+1}$ or, hence, any other puncture. The arc $a_{j_i}$ further subdivides this into two more regions, one of which lies inside $P$. Let $K$ be that region (the upper right quarter of $\head(F_{j_i})$ in Figure~\ref{fig:intersect-in-P}).

If $g_i$ leaves $P$ prior to intersecting $g_{i-1}$, then $g_i$ intersects $a_{j_{i+1}}$ before the intersection with $g_{i-1}$.  Thus $a_{j_{i+1}}$ enters the region $K$ via the arc $g_i$. It must then proceed to leave $K$, but it cannot cross $a_{j_i}$ by Step~1, and it cannot cross $g_{i-1}$ prior to $g_{i-1}$ intersecting $a_{j_i}$ or it would contradict the definition of $g_i$. So $a_{j_{i+1}}$ must both enter and exit $K$ by crossing $g_i$, forming a bigon in $K$ between $a_{j_{i+1}}$ and $g_i$, and contradicting minimal position.
\end{proof}

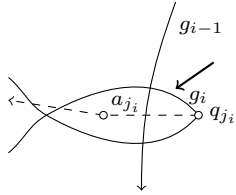
\begin{figure}[h]
\begin{tikzpicture}

\draw (3,1) edge[out=130,in=30] (1,1);
\draw (1,1) edge[out=210,in=40] (0.5,0.5);
\node at (3,1) [above] {\tiny $g_i$};

\draw (3,1) edge[out=230,in=330] (1,1);
\draw (1,1) edge[out=150,in=320] (0.5,1.5);
\node at (2.05,0.9) [above] {\tiny $a_{j_i}$};

\draw (3,1) edge[dashed] (1.75,1);
\draw (1.75,1) edge[dashed,->] (.5,1.2);

\draw (2.7,2.5) edge[out=250,in=90,->] (2.25,0);
\node at (2.6,2.4) [below right] {\tiny $g_{i-1}$};

\draw (3.2,1.7) edge[thick,->] (2.7,1.35);

\draw[fill=white] (1.75,1) circle (0.05cm);
\draw[fill=white] (3,1) circle (0.05cm);
\node at (3,1) [right] {\tiny $q_{j_i}$};

\end{tikzpicture}
\caption{If $g_i$ leaves $P$ prior to intersecting $g_{i-1}$, then there is an intersection between $g_i$ and $a_{j_{i+1}}$ where indicated by the arrow.} \label{fig:intersect-in-P}
\end{figure}

\noindent \textbf{Definition of $R$.} Two consecutive arcs $g_i$ and $g_{i+1}$ in the cycle divide $D$ into two regions. Let $R_i\subset D$ be the closed region bounded by $g_i$ and $g_{i+1}$ which does not contain the punctures $q_{j_i}$ and $q_{j_{i+1}}$ (See Figure~\ref{fig:p-prime}). Define $R = \bigcup R_i$. The polygon~$P$ also divides the disc into two regions. Let $E\subset D$ be the exterior of $P$, the region adjacent to $p$.

\medskip
\noindent \textbf{Step 4.} \emph{$E \subset R$.}

\begin{proof}

Since $E$ is connected, it suffices to show that $E \cap R$ is both open and closed in $E$. Each $R_i$ is closed by definition, so certainly $E \cap R$ is closed in $E$. Let $x \in E\cap R$. If $x$ does not lie on one of the arcs $g_i$, then $x$ lies in the interior of $R_i$ for some $i$, and therefore $x$ has an open neighbourhood in $R$.

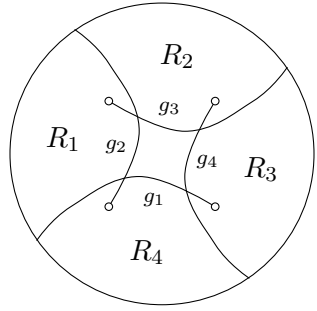
\begin{figure}[h]
\begin{tikzpicture}
\draw (2,2) circle (2cm);

\draw plot [smooth, tension=0.7] coordinates {(2.7,2.7) (2.3,1.7) (2.7,0.8) (3.15,0.35)};
\draw plot [smooth, tension=0.7] coordinates {(2.7,1.3) (1.7,1.7) (0.8,1.3) (0.35,0.85)};
\draw plot [smooth, tension=0.7] coordinates {(1.3,1.3) (1.7,2.3) (1.3,3.2) (0.85,3.65)};
\draw plot [smooth, tension=0.7] coordinates {(1.3,2.7) (2.3,2.3) (3.2,2.7) (3.65,3.15)};

\node at (1.9,1.4) {\tiny $g_1$};
\node at (1.4,2.1) {\tiny $g_2$};
\node at (2.1,2.6) {\tiny $g_3$};
\node at (2.6,1.9) {\tiny $g_4$};

\node at (0.7,2.2) {\small $R_1$};
\node at (2.2,3.3) {\small $R_2$};
\node at (3.3,1.8) {\small $R_3$};
\node at (1.8,0.7) {\small $R_4$};

\draw[fill=white] (2.7,2.7) circle (0.05cm);
\draw[fill=white] (2.7,1.3) circle (0.05cm);
\draw[fill=white] (1.3,1.3) circle (0.05cm);
\draw[fill=white] (1.3,2.7) circle (0.05cm);
\end{tikzpicture}
\caption{The arcs $g_i$, and the regions $R_i$ that they form. In more complicated examples, these $R_i$ may overlap.}\label{fig:p-prime}
\end{figure}

Suppose $x$ lies on $g_i$ for some $i$. By Steps~2 and~3, the arc $g_i$ intersects both $g_{i-1}$ and $g_{i+1}$ prior to leaving $P$ for the first time. Since $x \in E$, its position on the arc $g_i$ is after the intersections of $g_i$ with $g_{i-1}$ and $g_{i+1}$. Therefore, a sufficiently small neighbourhood of $x$ lies in the union $R_{i-1} \cup R_i \subset R$. It follows that $E \cap R$ is open in $E$, and hence $E \subset R$.
\end{proof}

\noindent \textbf{Step 5.} \emph{$E$ contains the puncture $q$.}

\begin{proof}

Suppose not. Then $q$ lies inside the polygon~$P$. By Lemma~\ref{lemma:tails} no arc from $q$ may pass through the head of any fish with $q$ in its tail. In particular, no arc from $q$ may pass through $F_i$ for any $i$. However, $a_i \in \head(F_i)$ for each $i$, and so no arc from $q$ may cross $P$. It follows that there are no arcs in $\mathcal{A}$ from~$q$ to~$p$. For the same reason, it would not be possible to add an arc from~$q$ to~$p$ to the collection~$\mathcal{A}$.

Now consider the space obtained by removing every arc in $\mathcal{A}$ from~$D$, and let $C$ be the component of that space which contains the puncture~$q$. Let $\alpha$ be an embedded arc in $C$ from $q$ to the boundary of $C$. The point $x \in C$ where $\alpha$ ends lies on some arc $\gamma \in \mathcal{A}$. Let $\beta$ be the subarc of $\gamma$ running from $x$ to $p$. Then up to homotopy, the arc obtained by concatenating $\alpha$ and $\beta$ is an arc from $q$ to $p$ on $D$ which intersects each arc in $\mathcal{A}$ at most once. Therefore, it is possible to add an arc from $q$ to $p$ to the collection $\mathcal{A}$, a contradiction.
\end{proof}

\noindent \textbf{Final contradiction.} By Step~5, the puncture $q$ lies in some $R_i$. Without loss of generality, assume $q \in R_1$, the region bounded by $g_1$ and $g_2$ (see Figure~\ref{fig:p-prime}). Recall that $q_{j_2} \notin R_1$. We claim that $q_{j_2 + 1} \in R_1$. Indeed, the arc $g_1$ by definition intersects $a_{j_2}$ when first leaving~$P$. Since $g_2$ cannot intersect $a_{j_2}$, because $a_{j_2} \in \head(F_{j_2})$, the region $R_1$ contains the half of $a_{j_2}$ which ends at $q_{j_2+1}$.

Now consider the arc $\out{j_2}$. It starts at $q_{j_2}$, and must be positioned such that $q_{j_2+1} \in \head(F_{j_2})$. For this to be possible, $\out{j_2}$ must enter $R_1$ by crossing $g_1$, then exit by intersecting $g_2$. After this it is not possible for $\out{j_2}$ to reenter $R_1$ without any double intersections, and so $\tail(F_{j_2}) \cap R_1 = \emptyset$. Therefore $q \notin \tail(F_{j_2})$, a contradiction. This finishes the proof of Lemma~\ref{lemma:cycles}.

\subsection{Proof of Lemma~\ref{lemma:passing}}

To prove Lemma~\ref{lemma:passing} we start from the following.

\begin{lemma}\label{lemma:order-fixed}
Let $r\in Q-\{v\}$ and consider arcs in $\mathcal{A}$ from $r$ to $p$ satisfying $\alpha_1 < \alpha_2 < \dotsb < \alpha_m < \alpha_1$ in the cyclic order about $r$.
If there is no fish $F=(\alpha_i,\alpha_j)$ with $h(F) = \{v\}$, then ${\tilde \alpha_1 \leq \tilde \alpha_2 \leq  \dotsb \leq  \tilde \alpha_m \leq  \tilde \alpha_1}$.
\end{lemma}

We recall that $\tilde \alpha_i$ is the image if $\alpha_i$ under the embedding $\widetilde D\subset D$.
Here the cyclic order on $\tilde \alpha_i$ is the one obtained after putting them in minimal position on $\widetilde D$.

\begin{proof}
We first note that if the arcs $\tilde \alpha_i$ are in minimal position, then the order of arcs is unchanged by the removal of $v$ (though some of them might become homotopic). We may thus assume that $\tilde \alpha_i$ are not in minimal position. We will perform a series of homotopies of the arcs $\tilde \alpha_i$ with support away from $\tilde r$, and with each homotopy reducing the total number of intersections among the collection $\{\tilde \alpha_i\}$ by one. Hence after finitely many such homotopies, the arcs will be in minimal position, and because every homotopy has support away from $\tilde r$, the order of the arcs about $\tilde r$ will remain unchanged.

For an intersecting pair $(\tilde \alpha_i, \tilde \alpha_j)$ (we will not call it a fish since the arcs might become disjoint after homotopy), let $\tail(\tilde \alpha_i, \tilde \alpha_j)$ be the region they bound adjacent to $\tilde p$. In order to perform our homotopy, we will first show the existence of a pair $(\tilde \alpha_i, \tilde \alpha_j)$ whose $T=\tail(\tilde \alpha_i, \tilde \alpha_j)$ is a half-bigon with the property that for every other arc $\tilde \alpha_k$ with $k\neq i, j$, if $\tilde \alpha_k$ enters $T$, then it exits before arriving at $\tilde p$.

Since $\tilde \alpha_i$ are not in minimal position, there is a half-bigon formed by some $\tilde \alpha_{i_1},\tilde \alpha_{i_2}$. Such a half-bigon is not adjacent to $\tilde r$, because then $F=(\alpha_{i_1},\alpha_{i_2})$ would satisfy $h(F) = \{v\}$, violating our assumption. Thus the half-bigon is $T_1=\tail(\tilde \alpha_{i_1}, \tilde \alpha_{i_2})$.

Suppose $(\tilde \alpha_{i_1},\tilde \alpha_{i_2})$ above is not the desired pair, that is, there is some arc $\tilde \alpha_{i_3}$ which enters $T_1$ and does not exit it prior to arriving at~$\tilde p$. Without loss of generality, suppose $\tilde \alpha_{i_2}$ is the arc which $\tilde \alpha_{i_3}$ intersects upon entering $T_1$. Then $T_2= \tail(\tilde \alpha_{i_3}, \tilde \alpha_{i_2})$ is properly contained in $T_1$.
If the pair $(\tilde \alpha_{i_3}, \tilde \alpha_{i_2})$ is still not the desired pair, we repeat the process, at each step obtaining $T_k \subsetneq T_{k-1}$.
As there are only finitely many arcs in the collection, eventually we must terminate at the desired pair $(\tilde \alpha_i, \tilde \alpha_j)$.

Let $U\subset \widetilde D$ be a small open neighbourhood of $T$, not containing any punctures, and not intersecting any arc other than $\tilde \alpha_i, \tilde \alpha_j$, and arcs which intersect $T$. We homotope  $\tilde \alpha_i$ within $U$ so that $\tilde \alpha_i \cap \tilde \alpha_j = \emptyset$, and such that for each $k \neq i, j$, the value $|\tilde \alpha_k \cap \tilde \alpha_i|$ remains unchanged. We note that after applying this homotopy, there are still no half-bigons adjacent to $\tilde r$, and we can repeat the process.
\end{proof}

\begin{cor}\label{lemma:vanishing}
Let $(\alpha, \beta)$ and $(\gamma, \delta)$ be two non-vanishing minimal $q$-fish with common nose $r$.
Then ${\tilde\alpha < \tilde\beta \leq \tilde\gamma < \tilde\delta \leq \tilde\alpha}$.
\end{cor}

\begin{proof}
Since $(\alpha, \beta)$ and $(\gamma, \delta)$ are minimal $q$-fish, by Remark~\ref{rmk:not_three} we have ${\alpha < \beta < \gamma < \delta < \alpha}$ in the cyclic order about $r$.
We will prove that the four arcs $\alpha, \beta, \gamma, \delta$ satisfy the condition of Lemma~\ref{lemma:order-fixed} that there is no fish $F$ formed by these four arcs with $h(F) = \{v\}$.

We identify the interior of $D$ with the interior of a punctured annulus~$A$ whose boundary circles correspond to $r$ and $p$. Consider the infinite cyclic cover $\overline A$ of $A$ corresponding to the kernel of the map $\pi_1A\to\Z$ induced by removing all the punctures of $A$ (except $r$ and $p$).
We represent $\overline A$ as an infinite strip with boundary lines $\bar r, \bar p$ corresponding to $r,p$ as depicted in Figure~\ref{fig:beta-gamma}. If $\overline \alpha,\overline \beta$ are intersecting lifts of $\alpha,\beta$ to~$\overline A$, we denote by $\head(\overline \alpha,\overline \beta)$ the region they bound in $\overline A$ adjacent to $\bar r$ (the lift of $\head(\alpha,\beta)$), and by $\tail(\overline \alpha,\overline \beta)$ the region they bound in $\overline A$ adjacent to $\bar p$ (the lift of $\tail(\alpha,\beta)$).

Let $\bar q$ be a lift of $q$ to $\overline A$. Consider the lifts $\overline\alpha, \overline\beta, \overline\gamma, \overline\delta$ such that $\overline\alpha$ intersects $\overline\beta$, $\overline\gamma$ intersects $\overline\delta$, and $\bar q$ is in both $\tail(\overline\alpha, \overline\beta)$ and $\tail(\overline\gamma, \overline\delta)$. Without loss of generality, they are ordered $\overline\alpha < \overline\beta < \overline\gamma <\overline \delta$ in $\overline A$ according to their starting points at $\bar r$. Since $\bar q$ is to the left of $\overline \alpha$ and to the right of $\overline \delta$, we have that $\overline\alpha$ intersects $\overline\delta$. Consequently $(\alpha, \delta)$ is a fish, with $\beta, \gamma$ starting in $\head(\alpha, \delta)$.

We claim that if there is some fish $F$ formed by $\alpha, \beta, \gamma, \delta$, with $h(F) = \{v\}$, then in particular $(\beta, \gamma)$ is such a fish. To justify the claim, suppose first that we have any three arcs $\alpha_1,\alpha_2,\alpha_3$ such that $(\alpha_1, \alpha_3)$ is a fish, and $\alpha_2$ begins in $\head(\alpha_1, \alpha_3)$. Observe that if $h(\alpha_1,\alpha_3) = \{v\}$, then either $(\alpha_1, \alpha_2)$ is a fish with $h(\alpha_1, \alpha_2) = \{v\}$, or $(\alpha_2, \alpha_3)$ is a fish with $h(\alpha_2, \alpha_3) = \{v\}$.

Now we can apply this observation to the configuration in the claim. If ${h(\alpha, \delta) = \{v\}}$, then one of $(\alpha, \beta)$ or $(\beta, \delta)$ is a fish $F$ with $h(F) = \{v\}$. By assumption, $(\alpha, \beta)$ is a non-vanishing fish, so we must have $h(\beta, \delta) = \{v\}$. Repeating the same argument for the triple $\beta,\gamma,\delta$, we conclude that $(\beta, \gamma)$ is a fish, and $h(\beta, \gamma) = \{v\}$, as desired. Analogously, if $(\alpha, \gamma)$ is a fish with $h(\alpha,\gamma) = \{v\}$, then also $(\beta, \gamma)$ is a fish with $h(\beta, \gamma) = \{v\}$, as desired.

Finally, if $(\gamma,\beta),(\gamma,\alpha),$ or $(\delta,\beta)$ were a fish $F$ with $h(F) = \{v\}$, then, arguing analogously as in the preceding paragraph, the same would be true for $(\delta,\alpha)$, which is not even a fish (wrong order of arcs). This justifies the claim. Thus the arcs $\overline \beta$ and $\overline \gamma$ intersect, see Figure~\ref{fig:beta-gamma}.

\begin{figure}
\begin{tikzpicture}

\draw (0,0) -- (8,0);
\draw (0,2) -- (8,2);
\draw [dashed] (0,0) -- (0,2);
\draw [dashed] (8,0) -- (8,2);

\node at (4,2) [above] {\tiny $\bar r$};
\node at (4,0) [below] {\tiny $\bar p$};

\draw (2,2) -- (2,1.75);
\node at (2,2) [below left] {\tiny $\overline\alpha$};

\draw (3,2) edge[out=-90, in=90] (5,0);
\node at (3,2) [below left] {\tiny $\overline\beta$};

\draw (5,2) edge[out=-90, in=90] (3,0);
\node at (5,2) [below right] {\tiny $\overline\gamma$};

\draw (6,2) -- (6,1.75);
\node at (6,2) [below right] {\tiny $\overline\delta$};

\end{tikzpicture}
\caption{Since $\overline \beta$ and $\overline \gamma$ intersect, we have $\tail(\overline\alpha, \overline\beta) \cap \tail(\overline\gamma, \overline\delta) \subset \head(\overline\beta, \overline\gamma)$.}
\label{fig:beta-gamma}
\end{figure}
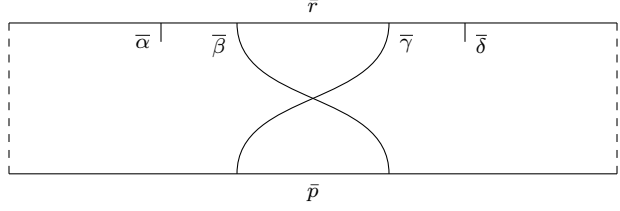

In $\overline A$, observe that $\tail(\overline\alpha, \overline\beta)$ lies to the right of~$\overline\beta$. Similarly, $\tail(\overline\gamma, \overline\delta)$ lies to the left of~$\overline \gamma$. Hence their intersection, and $\bar q$ which lies within it, lies in the region which is both right of~$\overline \beta$ and left of~$\overline \gamma$. This region is $\head(\overline \beta, \overline \gamma)$. Therefore $h(\beta, \gamma) \neq \{v\}$, since it must at least contain~$q$, contradiction.

This proves that no fish $F$ formed by $\alpha, \beta, \gamma, \delta$ has $h(F) = \{v\}$, and allows us to apply Lemma~\ref{lemma:order-fixed}. The strict inequalities follow from the fact that $(\alpha, \beta)$ and $(\gamma, \delta)$ are non-vanishing.
\end{proof}

We are now ready for the following.

\begin{proof}[Proof of Lemma~\ref{lemma:passing}]
Let $F=(\alpha, \beta)\in \mathcal F'_q$ with any nose $r$. In the case where the fish $\widetilde F=(\tilde \alpha, \tilde \beta)$ is minimal, we simply define $\varphi(F) = \widetilde F$.

Suppose now that the fish $\widetilde F$ is not minimal. This implies $\tilde \alpha$ and $\tilde \beta$ are no longer consecutive (after putting the arcs in $\widetilde D$ in minimal position), so there must be an arc $\tilde \varepsilon\in \widetilde{\mathcal{A}}$ that emanates from $\tilde r$ in $\head(\widetilde F)$. The arc $\tilde \varepsilon$ must intersect either $\tilde\alpha$ or $\tilde \beta$ before proceeding to $\tilde p$. Without loss of generality, suppose it intersects $\tilde \alpha$ when leaving $\head(\widetilde F)$.

If $\tilde \varepsilon$ does not intersect $\tilde \beta$, then $(\tilde \alpha, \tilde \varepsilon)$ is a $\tilde q$-fish. If $\tilde \varepsilon$ does intersect $\tilde \beta$, then either $(\tilde \alpha, \tilde \varepsilon)$ or $(\tilde \varepsilon, \tilde \beta)$ is a $\tilde q$-fish (see Figure~\ref{fig:epsilon}). If this resulting fish is minimal, then we finish our process. If it is not minimal, then we note that there are fewer arcs in $\widetilde{\mathcal{A}}$ originating in the head of either $(\tilde \alpha, \tilde \varepsilon)$ or $(\tilde \varepsilon, \tilde \beta)$ than there were in $(\tilde \alpha, \tilde \beta)$. Therefore if we continue this process, it will terminate after finitely many steps with a minimal $\tilde q$-fish that we define to be $\varphi(F)$.

\begin{figure}
\begin{tikzpicture}
\draw (1.5,1.5) circle (1.5cm);
\draw (5.5,1.5) circle (1.5cm);

\draw plot [smooth, tension=0.7] coordinates {(1.5,2) (2,1.4) (1.5,0.8) (0.8,0.2)};
\draw plot [smooth, tension=0.7] coordinates {(1.5,2) (1,1.4) (1.5,0.8) (2.2,0.2)};
\draw (1.5,2) edge [out=315,in=45] (0.5,0.35);

\draw [fill=white] (1.5,0.5) circle (0.05cm) node [below] {\tiny $\tilde q$};

\draw plot [smooth, tension=0.7] coordinates {(5.5,2) (6,1.4) (5.5,0.8) (4.8,0.2)};
\draw plot [smooth, tension=0.7] coordinates {(5.5,2) (5,1.4) (5.5,0.8) (6.2,0.2)};
\draw plot [smooth, tension=0.7] coordinates {(5.5,2) (5.5,1.5) (4.75,0.8) (5.8,0.05)};

\draw [fill=white] (5.5,0.5) circle (0.05cm);
\draw [fill=white] (5.2,0.2) circle (0.05cm);

\draw[fill=white] (1.5,2) circle (0.05cm) node [left] {\tiny $\tilde\alpha$} node [right] {\tiny $\tilde\beta$};
\draw[fill=white] (5.5,2) circle (0.05cm) node [left] {\tiny $\tilde\alpha$} node [right] {\tiny $\tilde\beta$};

\node at (5.4,1.9) [below] {\tiny $\tilde\varepsilon$};
\node at (1.5,1.9) [below] {\tiny $\tilde\varepsilon$};
\end{tikzpicture}
\caption{On the left, $\tilde \varepsilon$ intersects only $\tilde \alpha$ and this forms a $\tilde q$-fish ``smaller'' than $(\tilde \alpha, \tilde \beta)$. On the right, $\tilde \varepsilon$ also intersects $\tilde \beta$ and, depending on which puncture is $\tilde q$, either $(\tilde \alpha,\tilde \varepsilon)$ or $(\tilde \varepsilon, \tilde\beta)$ is the desired $\tilde q$-fish $\varphi(\tilde\alpha, \tilde\beta)$.}
\label{fig:epsilon}
\end{figure}
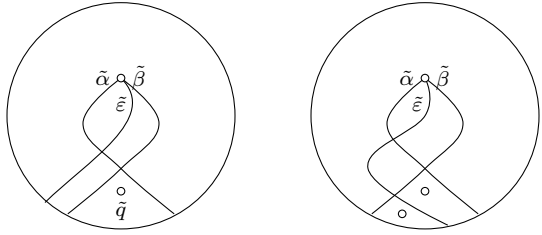

Let $F = (\alpha, \beta)$ and $F' = (\gamma, \delta)$ be two distinct non-vanishing minimal $q$-fish. Clearly if they have distinct noses, then $\varphi(F)$ and $\varphi(F')$ will have also distinct noses, in particular they will be distinct. Thus we consider only the case where $F$ and $F'$ have the same nose $r$.

By Corollary~\ref{lemma:vanishing}, we have ${\tilde\alpha < \tilde\beta \leq \tilde\gamma < \tilde\delta \leq \tilde\alpha}$.
Moreover, by construction $\varphi(\alpha, \beta)$ (respectively $\varphi(\gamma, \delta)$) is a fish whose head intersects a neighbourhood of~$r$ between $\tilde \alpha$ and $\tilde \beta$ (respectively between $\tilde \gamma$ and $\tilde \delta$). Hence $\varphi(\alpha, \beta) \neq \varphi(\gamma, \delta)$, as required.
\end{proof}

\section{Counting Excess Arcs using Tail Punctures of Fish} \label{section:bigons}

\begin{proof}[Proof of Proposition \ref{lemma:bigons}]
For simplicity of notation, let $k = k_r$. Let $\alpha_1 < \alpha_2 < \dotsc < \alpha_k< \alpha_1$
be the cyclic order about $r$ of the arcs in $\mathcal A$ with nose $r$.

We identify the interior of $D$ with the interior of a punctured annulus~$A$ whose boundary circles correspond to $r$ and $p$. Consider the infinite cyclic cover $\overline A$ of $A$ corresponding to the kernel of the map $\pi_1A\to\Z$ induced by removing all the punctures of $A$ . Let $\overline Q$ be the punctures of $\overline A$. We represent $\overline A=\R\times (0,1)-\overline Q$, where $\Z$ acts as horizontal translations.

\begin{figure}
\begin{tikzpicture}

\draw (2,2) circle (2cm);

\draw (2,2) -- (3.42,0.56);
\node at (3,1) [right] {\tiny $\alpha_1$};

\draw (2,2) edge[out=0,in=270] (3.3,3.5);
\node at (2.8,2) {\tiny $\alpha_2$};

\draw (2,2) edge[out=90,in=180] (3.95,2.4);
\node at (2.5,2.65) {\tiny $\alpha_3$};

\draw (2,2) -- (0.58,3.42);
\node at (1,3) [left] {\tiny $\alpha_4$};

\draw (1.5,2) edge[out=180,in=-135] (1.25,2.75);
\draw (1.25,2.75) edge[out=45,in=180] (3.72,3);
\node at (0.95,2.4) {\tiny $\alpha_5$};

\draw[fill=white] (2,2) circle (0.5cm) node [below left] {\tiny $r$};

\end{tikzpicture}
\end{figure}

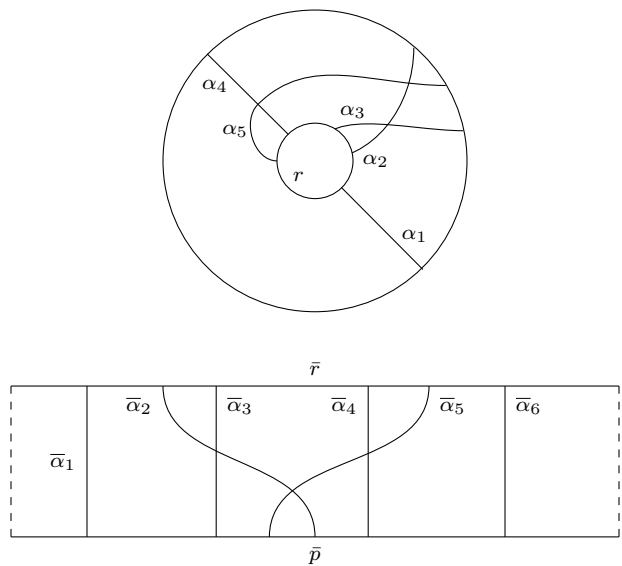
\begin{figure}
\begin{tikzpicture}

\draw (0,0) -- (8,0);
\draw (0,2) -- (8,2);
\draw [dashed] (0,0) -- (0,2);
\draw [dashed] (8,0) -- (8,2);

\node at (4,2) [above] {\tiny $\bar r$};
\node at (4,0) [below] {\tiny $\bar p$};

\draw (1,2) -- (1,0);
\node at (1,1) [left] {\tiny $\overline{\alpha}_1$};

\draw (2,2) edge[out=-90,in=90] (4,0);
\node at (2,2) [below left] {\tiny $\overline{\alpha}_2$};

\draw (2.7,2) -- (2.7,0);
\node at (2.7,2) [below right] {\tiny $\overline{\alpha}_3$};

\draw (4.7,2) -- (4.7,0);
\node at (4.7,2) [below left] {\tiny $\overline{\alpha}_4$};

\draw (5.5,2) edge[out=-90,in=90] (3.4,0);
\node at (5.5,2) [below right] {\tiny $\overline{\alpha}_5$};

\draw (6.5,2) -- (6.5,0);
\node at (6.5,2) [below right] {\tiny $\overline{\alpha}_6$};

\end{tikzpicture}
\caption{An example lifting of arcs from a punctured annulus to the infinite strip. Here $\overline{\alpha}_6$ is a second lift of~$\alpha_1$.}
\label{fig:lifts}
\end{figure}

Let $\mathcal V$ be the set of homotopy classes of arcs in $\overline A$ that are lifts of arcs from $r$ to $p$ in $A$. We define the following function $d\colon \mathcal V\times \mathcal V\to \Z$. Let $\hat A$ be the space obtained from $\overline A$ by introducing two points $r_\infty, p_\infty$ at infinity. We declare that the basis neighbourhoods in $\hat A$ of $r_\infty, p_\infty$ are the unions of these points with horizontal strips $\R\times(1,1-\epsilon), \R\times(0,\epsilon)$ disjoint from $\overline Q$.

Let $\psi \colon H_1(\hat A,\Z)\to \Z$ be the map on homology determined by the property that for each counterclockwise oriented circle $c$ around a single puncture in $\overline Q$ we have $\psi (c)=1$. For each pair $\gamma,\delta\in \mathcal V$ we define $d(\gamma,\delta)$ as follows. Let $\hat \gamma,\hat \delta$ be 1-chains in $\hat A$ obtained from $\gamma,\delta$ by compactifying using $r_\infty$ and $p_\infty$. Then $\hat \gamma -\hat \delta$ is a $1$-cycle. We put $d(\gamma,\delta)=\psi(\hat \gamma -\hat \delta)$. Since $\psi$ is a homomorphism, we have an additivity property: for any $\beta, \gamma,\delta\in \mathcal V$ we have $d(\beta, \gamma) + d(\gamma, \delta) = d(\beta, \delta)$.

We lift the arcs in $\mathcal A$ with nose $r$ to $\overline A$. Each such lift is an arc from $\bar r$ to $\bar p$.
Choose $\overline \alpha_1$ to be one of the lifts of $\alpha_1$. Notice that consecutive arcs along $\bar r$ in $\overline A$ are lifts of consecutive arcs in $A$ about $r$. Hence we may label arcs in $\overline A$ by $(\dotsc, \overline \alpha_1, \overline \alpha_2, \dotsc, \overline \alpha_{k+1}, \dotsc)$ where $\overline \alpha_{mk+i}$ is a lift of $\alpha_i$ for $1 \leq i \leq k$ and $m \in \Z$. In particular, $\overline \alpha_{k+1}$ is a lift of $\alpha_1$. See Figure~\ref{fig:lifts}.

Note that for $i<j$ the value $d(\overline \alpha_i,  \overline \alpha_j)$ can be easily described in the following way. If $\overline\alpha_i$ and $\overline \alpha_j$ are disjoint, then $d(\overline \alpha_i, \overline \alpha_j)$ is the number of punctures in $\overline Q$ lying in the bounded region of $\overline A$ between $\overline\alpha_i$ and $\overline \alpha_j$. If $\overline\alpha_i$ and $\overline \alpha_j$ intersect, then
$d(\overline \alpha_i, \overline \alpha_j)=|h(\overline \alpha_i, \overline \alpha_j)|-|t(\overline \alpha_i, \overline \alpha_j)|$.
Since $\alpha_1$ does not self-intersect, we have that its two lifts $\overline \alpha_1$ and $\overline \alpha_{k+1}$ are disjoint, and it is easy to see that $d(\overline \alpha_1, \overline\alpha_{k+1}) = |Q - \{r\}| = |\chi|$.

For $1 \leq i \leq k$, if the arcs $\alpha_{i}$ and $\alpha_{i+1}$ are disjoint (and hence do not form a minimal fish), then ${d(\overline \alpha_i, \overline \alpha_{i+1}) \geq 1}$, as there must be at least one puncture between them. If $\alpha_{i}$ and $\alpha_{i+1}$ intersect, then they form a minimal fish $F_i$, and we have $d(\overline \alpha_i, \overline \alpha_{i+1}) \geq 1 - |t(F_i)|$ as there must be at least one puncture in $h(F_i)$. Using the additivity of $d$, and the convention $|t(F_i)| = 0$ whenever there was no fish $F_i$, we obtain:
$$\sum_{i=1}^{k} (1 - |t(F_i)|) \leq \sum_{i=1}^k d(\overline \alpha_i, \overline \alpha_{i+1}) = d(\overline \alpha_1, \overline \alpha_{k+1}) = |\chi|$$
Equivalently,
$$k \leq \sum_{F \in \mathcal{F}^r} |t(F)| + |\chi|.\qedhere$$

\end{proof}

\end{document}